\documentclass[12pt,a4paper]{amsart}
\usepackage{amsfonts}
\usepackage{amsthm}
\usepackage{amsmath}
\usepackage{amscd}
\usepackage[latin2]{inputenc}
\usepackage{t1enc}
\usepackage[mathscr]{eucal}
\usepackage{indentfirst}
\usepackage{graphicx}
\usepackage{graphics}
\usepackage{pict2e}
\usepackage{epic}
\usepackage{url}
\numberwithin{equation}{section}
\usepackage[margin=2.7cm]{geometry}

\theoremstyle{plain}
\newtheorem{Th}{Theorem}[section]
\newtheorem{Lemma}[Th]{Lemma}

 \theoremstyle{definition}

\newtheorem{Rem}[Th]{Remark}
\newtheorem{?}[Th]{Problem}

\begin{document}

\title[Sum of the smallest and largest eigenvalues of a triangle-free graph]{Note on the sum of the smallest and largest eigenvalues of a triangle-free graph}

\author[P. Csikv\'ari]{P\'{e}ter Csikv\'{a}ri}

\address{Alfr\'ed R\'enyi Institute of Mathematics \&  E\"{o}tv\"{o}s Lor\'{a}nd University \\ Department of Computer 
Science} 

\email{peter.csikvari@gmail.com}

\thanks{The  author  is  supported by the Counting in Sparse Graphs Lend\"ulet Research Group.}



\begin{abstract} Let $G$ be a triangle-free graph on $n$ vertices with adjacency matrix eigenvalues $\mu_1(G)\geq \mu_2(G)\geq \dots \geq \mu_n(G)$. In this paper we study the quantity
$$\mu_1(G)+\mu_n(G).$$
We prove that for any triangle-free graph $G$ we have
$$\mu_1(G)+\mu_n(G)\leq (3-2\sqrt{2})n.$$
This was proved for regular graphs by Brandt, we show that the condition on regularity is not necessary. We also prove that among triangle-free strongly regular graphs the Higman-Sims graph achieves the maximum of 
$$\frac{\mu_1(G)+\mu_n(G)}{n}.$$

\end{abstract}

\maketitle

\section{Introduction} In this paper every graph is simple. Motivated by the papers \cite{Brandt} and \cite{LNO} we study the following problem.
Let $\mathcal{G}_3$ be  the family of  triangle-free graphs, and for a graph $G$ on $v(G)=n$ vertices let $\mu_1(G)\geq \mu_2(G)\geq \dots \geq \mu_n(G)$ be the eigenvalues of the adjacency matrix of $G$. The problem is to determine
$$c_3=\sup_{G\in \mathcal{G}_3}\frac{\mu_1(G)+\mu_n(G)}{v(G)}.$$

Brandt \cite{Brandt} proved that for regular triangle-free graphs we have
$$\mu_1(G)+\mu_n(G)\leq (3-2\sqrt{2})n.$$
Very recently Balogh, Clemen, Lidick\'y, Norin and Volec proved that for regular 
triangle-free graphs we have $\mu_1(G)+\mu_n(G)\leq \frac{15}{94}n<0.1596n$ and they mention in their paper that a similar but larger computation also gives the result $0.15467$ instead of $0.1596$. In fact, they study the smallest eigenvalue $q_n(G)$ of the so-called signless laplacian matrix $\overline{L}=D+A$, where $D$ is is the diagonal matrix consisting of the degrees of the vertices and $A$ is the adjacency matrix of the graph $G$. The quantity $q_n(G)$ coincides with $\mu_1(G)+\mu_n(G)$ if $G$ is regular.  
Balogh, Clemen, Lidick\'y, Norin and Volec mentions that in case of regular graphs they can further improve their result to prove $q_n(G)\leq 0.15442n$.
Our first result is to prove that in Brandt's theorem one can drop the condition of regularity.

\begin{Th} \label{Th1} Let $G$ be a triangle-free graph on $n$ vertices, and let $\mu_1(G)\geq \mu_2(G)\geq \dots \geq \mu_n(G)$ be the eigenvalues of its adjacency matrix. Then
$$\mu_1(G)+\mu_n(G)\leq (3-2\sqrt{2})n.$$
\end{Th}

The proof of Theorem~\ref{Th1} heavily relies on the following lemma which might be of independent interest.

\begin{Lemma} \label{Lem} Let $G$ be a triangle-free graph on $n$ vertices, and let $\mu_1(G)\geq \mu_2(G)\geq \dots \geq \mu_n(G)$ be the eigenvalues of its adjacency matrix. Then
$$\mu_1(G)\leq \frac{-n\mu_n(G)}{\mu_1(G)-\mu_n(G)}.$$
\end{Lemma}

Brandt \cite{Brandt} also realized that  for the so-called Higman-Sims graph $H$ we have
$$\frac{\mu_1(H)+\mu_n(H)}{v(H)}=\frac{22+(-8)}{100}=0.14$$
which gives a rather good lower bound for $c_3$. Higman-Sims graph is the unique strongly regular graph with parameters $(100,22,0,6)$. Recall that a graph $G$ is a strongly regular graph with paremeters $(n,k,a,b)$ if it has $n$ vertices, $k$-regular, any two adjacent vertices have exactly $a$ common neighbors, and any two non-adjacent vertices have exactly $b$ common neighbors.  In this paper we show that among the strongly regular graphs, it is indeed the Higman-Sims graph which maximizes the quantity
$$\frac{\mu_1(G)+\mu_n(G)}{v(G)}.$$

Note that only finitely many triangle-free strongly regular graphs are known currently, but we do not rely on this fact.

\begin{Th} \label{Th2} Let $G$ be a triangle-free strongly regular graph on $n$ vertices. Then
$$\mu_1(G)+\mu_n(G)\leq 0.14n$$
with equality if and only if $G$ is the Higman-Sims graph.
\end{Th}

\section{Proof of Theorem~\ref{Th1}}

We begin with proving Lemma~\ref{Lem}. Before we actually start it let us mention that for regular graphs this lemma is a simple consequence of the Hoffman-Delsarte bound for independent sets. Indeed, let $\alpha(G)$ denote the size of the largest independent set of a $d$--regular graph. Then by the Hoffman-Delsarte bound we have
$$\alpha(G)\leq \frac{-n\mu_n(G)}{d-\mu_n(G)}.$$
Since $G$ is triangle-free, the neighbors of a vertex determine an independent set, whence $d\leq \alpha(G)$. Since $d=\mu_1(G)$ we get that
$$\mu_1(G)=d\leq \alpha(G)\leq \frac{-n\mu_n(G)}{d-\mu_n(G)}\leq \frac{-n\mu_n(G)}{\mu_1(G)-\mu_n(G)}.$$
Based on this inequality Brandt proved that 
$$\mu_1(G)+\mu_n(G)\leq (3-2\sqrt{2})n.$$
So after proving Lemma~\ref{Lem} we practically copy the proof of Brandt.

\begin{proof}[Proof of Lemma~\ref{Lem}.]
Let $\mu_s,\dots ,\mu_n$ be the set of non-positive eigenvalues. Then
$$0=6\cdot \mathrm{number\, \,  of\, \, triangles}=\sum_{i=1}^n\mu_i^3\geq \mu_1^3+\sum_{i=s}^n\mu_i^3.$$
Hence
$$\sum_{i=s}^n(-\mu_i)^3\geq \mu_1^3.$$
On the other hand, we have
$$\sum_{i=s}^n(-\mu_i)^3\leq (-\mu_n)\sum_{i=s}^n(-\mu_i)^2\leq (-\mu_n)(2e(G)-\mu_1^2)\leq (-\mu_n)(n\mu_1-\mu_1^2).$$
Hence $\mu_1^3\leq (-\mu_n)(n\mu_1-\mu_1^2),$ thus
$\mu_1^2\leq (-\mu_n)(n-\mu_1),$
or in other words,
$$\mu_1\leq \frac{-n\mu_n}{\mu_1-\mu_n}.$$
\end{proof}

\begin{proof}[Proof of Theorem~\ref{Th1}] As we mentioned earlier this proof practically follows the argument of \cite{Brandt}.

We only need to solve the constrained maximization problem:
$$\max\left\{ \frac{\mu_1+\mu_n}{n}\ |\ \mu_1\leq \frac{-n\mu_n}{\mu_1-\mu_n}\right\}.$$
Let $a=\mu_1$, $b=-\mu_n$ then we have
$a\leq \frac{nb}{a-b}$
which is equivalent to 
$\frac{a^2}{n-a}\leq b.$
Hence
$$\frac{a-b}{n}\leq  \frac{1}{n}\left(a-\frac{a^2}{n-a}\right)=\frac{an-2a^2}{n(n-a)}.$$
So with the notation $\alpha=a/n$ we need to maximize
$f(\alpha):=\frac{\alpha-2\alpha^2}{1-\alpha}.$
Its derivative is 
$\frac{1-4\alpha+2\alpha^2}{(1-\alpha)^2}$
which is $0$ at $\alpha=1\pm 1/\sqrt{2}$. Note that $\mu_1\leq \Delta\leq n-1$, where $\Delta$ is the largest degree, so $0\leq \alpha\leq 1$. So we only need to consider $\alpha=1-1/\sqrt{2}$ and the extreme points of the interval, $\alpha=0$ and $1$, to see that $f(\alpha)$ is indeed maximal at $1-1/\sqrt{2}$ and $f(\alpha)=3-2\sqrt{2}$.

Hence $\mu_1+\mu_n\leq (3-2\sqrt{2})n$.

\end{proof}

\section{Proof of Theorem~\ref{Th2}}

In this secation we prove Theorem~\ref{Th2}.

\begin{proof}[Proof of Theorem~\ref{Th2}]
Suppose for contradiction that $G$ is a strongly regular graph with eigenvalues $(k,\mu_2^{(m_2)},\mu_n^{(m_n)})$ such that
$\frac{k+\mu_n}{n}>0.14.$
Let $-\mu_n=r$ and $r/k=x$. Again we use that
$k\leq \alpha(G)\leq \frac{-n\mu_n}{k-\mu_n}.$
Hence 
$\frac{r}{k}\geq \frac{k+r}{n}.$
Then
$$0.14<\frac{k-r}{n}=\frac{k-r}{k+r}\cdot \frac{k+r}{n}\leq \frac{k-r}{k+r}\cdot \frac{r}{k}=\frac{x(1-x)}{1+x}.$$
From which we get that $x>1/5$. Secondly, $m_n\geq \alpha(G)\geq k$ since we can assume that $\mu_2>0$. (Note that $\mu_2>0$ if $G$ is not a blow-up of a complete graph.)
Hence
$kn=2e(G)\geq m_n\mu_n^2\geq kr^2.$
So we have $n>r^2$. So we have two inequalities:
$0.14<\frac{k-r}{n}=\frac{k(1-x)}{n}$
and
$n>r^2.$
Then
$n>r^2=(kx)^2>x^2\frac{0.14^2}{(1-x)^2}n^2.$
Thus
$\frac{1}{0.14^2}\frac{(1-x)^2}{x^2}>n.$
Since $x>1/5$ we have 
$\frac{(1-x)^2}{x^2}<16.$
Hence
$n<\frac{16}{0.14^2}\approx 816.33.$
Now we can finish the proof since we know all possible strongly regular graph parameters up to $816$, see Brouwer's website \cite{Brouwer} and for the triangle-free strongly regular graph parameters the table on the next page. One can check that indeed the Higman-Sims graph achieves the maximum of $(\mu_1(G)+\mu_n(G))/n$.
\end{proof}

\begin{Rem} An interesting thing arises from the table on Andries Brouwer's website. If there were a strongly regular graph $G$ with parameters $(28,9,0,4)$ then for this graph $G$ we would have
$$\frac{\mu_1(G)+\mu_n(G)}{v(G)}=\frac{9+(-5)}{28}=\frac{1}{7}>0.14.$$
It is known that there is no such strongly regular graph just as there is no strongly regular graph with parameters $(64,21,0,10)$. For this graph we would have
$$\frac{\mu_1(G)+\mu_n(G)}{v(G)}=\frac{21+(-11)}{64}=\frac{10}{64}>\frac{1}{7}>0.14.$$

\end{Rem}

\begin{center}
$$\begin{array}{|c|c|c|c|c|c|c|c|c|c|c|} \hline
n & k & a & b & \vartheta_1 & \vartheta_2 & m_1 & m_2 & \frac{k+\vartheta_2}{n} &\textrm{Appr.} & \textrm{Existence} \\[6pt] \hline 
5 & 2 & 0 & 1 & \frac{\sqrt{5}-1}{2} &  \frac{-\sqrt{5}-1}{2}  & 2 & 2 & \frac{3-\sqrt{5}}{10} &0.076 &\textrm{Yes} \\[6pt] \hline
10 & 3 & 0 & 1 & 1 & -2 & 4 & 5 & \frac{1}{10} & 0.1& \textrm{Yes} \\[6pt] \hline
16 & 5 & 0 & 2 & 1 & -3 & 10 & 5 & \frac{2}{16} & 0.125&\textrm{Yes} \\[6pt] \hline
28 & 9 & 0 & 4 & 1 & -5 & 21 & 6 & \frac{4}{28} & 0.142&\textrm{No} \\[6pt] \hline
50 & 7 & 0 & 1 & 2 & -3 & 28 & 21 & \frac{4}{50} & 0.08 &\textrm{Yes} \\[6pt] \hline
56 & 10 & 0 & 2 & 2 & -4 & 35 & 20 & \frac{6}{56} & 0.106&\textrm{Yes} \\[6pt] \hline
64 & 21 & 0 & 10 & 1 & -11 & 56 & 7 & \frac{10}{64} & 0.156&\textrm{No} \\[6pt] \hline
77 & 16 & 0 & 4 & 2 & -6 & 55 & 21 & \frac{10}{77} &0.129 &\textrm{Yes} \\[6pt] \hline
100 & 22 & 0 & 6 & 2 & -8 & 77 & 22 & \frac{14}{100} & 0.14&\textrm{Yes} \\[6pt] \hline
162 & 21 & 0 & 3 & 3 & -6 & 105 & 56 & \frac{15}{162} & 0.092&\textrm{?} \\[6pt] \hline
176 & 25 & 0 & 4 & 3 & -7 & 120 & 55 & \frac{18}{176} &0.102 &\textrm{?} \\[6pt] \hline
210 & 33 & 0 & 6 & 3 & -9 & 154 & 55 & \frac{24}{210} &0.114 &\textrm{?} \\[6pt] \hline
266 & 45 & 0 & 9 & 3 & -12 & 209 & 56 & \frac{33}{266} &0.124 &\textrm{?} \\[6pt] \hline
324 & 57 & 0 & 12 & 3 & -15 & 266 & 57 & \frac{42}{324} &0.129 &\textrm{No} \\[6pt] \hline
352 & 26 & 0 & 2 & 4 & -6 & 208 & 143 & \frac{20}{352} & 0.056&\textrm{?} \\[6pt] \hline
352 & 36 & 0 & 4 & 4 & -8 & 231 & 120 & \frac{28}{352} &0.079 &\textrm{?} \\[6pt] \hline
392 & 46 & 0 & 6 & 4 & -10 & 276 & 115 & \frac{36}{392} & 0.091&\textrm{?} \\[6pt] \hline
552 & 76 & 0 & 12 & 4 & -16 & 437 & 114 & \frac{60}{552} &0.108 &\textrm{?} \\[6pt] \hline
638 & 49 & 0 & 4 & 5 & -9 & 406 & 231 & \frac{40}{638} &0.062 &\textrm{?} \\[6pt] \hline
650 & 55 & 0 & 5 & 5 & -10 & 429 & 220 & \frac{50}{650} & 0.076&\textrm{?} \\[6pt] \hline
667 & 96 & 0 & 16 & 4 & -20 & 551 & 115 & \frac{76}{667} & 0.113&\textrm{?} \\[6pt] \hline
704 & 37 & 0 & 2 & 5 & -7 & 407 & 296 & \frac{30}{704} & 0.042&\textrm{?} \\[6pt]\hline
784 & 116 & 0 & 20 & 4 & -24 & 667 & 116 & \frac{92}{784} &0.117 &\textrm{?} \\[6pt] \hline
800 & 85 & 0 & 10 & 5 & -15 & 595 & 204 & \frac{70}{800} & 0.087&\textrm{?} \\[6pt] \hline
\end{array}
$$

\end{center}

\end{document}